\documentclass[11pt]{amsart}
\usepackage{xcolor}
\usepackage{graphicx}	
\usepackage{amssymb}
\usepackage{amsmath}
\usepackage{amsopn}
\usepackage{color}
\usepackage{eurosym}
\newtheorem{definition}{Def{}inition}[section]

\newtheorem{theorem}[definition]{Theorem}
\newtheorem{proposition}[definition]{Proposition}

\newtheorem{remark}[definition]{Remark}

\newcommand{\bmu}{\bar{\mu}}
\newcommand{\bnu}{\bar{\nu}}

\newcommand{\btheta}{\bar{\theta}}

\newcommand{\ZZ}{{\mathbb{Z}}}
\newcommand{\CC}{{\mathbb{C}}}

\newcommand{\RR}{{\mathbb{R}}}

\newcommand{\CP}{{\mathbb{CP}}}

\newcommand{\id}{\mathrm{Id}}

\newcommand{\surj}{\twoheadrightarrow}
\newcommand{\inc}{\hookrightarrow}
\newcommand{\too}{\longrightarrow}

\newcommand{\la}{\langle}
\newcommand{\ra}{\rangle}
\newcommand{\bd}{\partial}
\newcommand{\x}{\times}

\newcommand{\GL}{\operatorname{GL}}
\newcommand{\SO}{\operatorname{SO}}
\newcommand{\U}{\operatorname{U}}
\newcommand{\Sp}{\operatorname{Sp}}
\newcommand{\cA}{\mathcal{A}}
\newcommand{\cC}{\mathcal{C}}
\newcommand{\Stab}{\operatorname{Stab}}
\newcommand{\diag}{\mathrm{diag}}

\begin{document}

\title[A complex, symplectic and non-K\"ahler $6$-manifold]{A $6$-dimensional simply connected complex 
and symplectic manifold with no K\"ahler metric}

\author[G. Bazzoni]{Giovanni Bazzoni}
\address{Fakult\"{a}t f\"{ur} Mathematik, Universit\"{a}t Bielefeld, Postfach 100301, D-33501 Bielefeld}
\email{gbazzoni@math.uni-bielefeld.de}

\author[M. Fern\'andez]{Marisa Fern\'{a}ndez}
\address{Universidad del Pa\'{\i}s Vasco,
Facultad de Ciencia y Tecnolog\'{\i}a, Departamento de Matem\'aticas,
Apartado 644, 48080 Bilbao, Spain}
\email{marisa.fernandez@ehu.es}

\author[V. Mu\~{n}oz]{Vicente Mu\~{n}oz}
\address{Facultad de Matem\'aticas, Universidad Complutense de Madrid, Plaza de Ciencias
3, 28040 Madrid, Spain}
\address{Instituto de Ciencias Matem\'aticas (CSIC-UAM-UC3M-UCM),
C/ Nicol\'as Cabrera 15, 28049 Madrid, Spain}
\email{vicente.munoz@mat.ucm.es}


\begin{abstract}
We construct a simply connected compact manifold which has complex and symplectic structures but does not admit K\"ahler metric, in the lowest 
possible dimension where this can happen, that is, dimension $6$. Such a 
manifold is automatically formal and has even odd-degree Betti numbers 
but it does not satisfy the Lefschetz property for any symplectic form.
\end{abstract}

\maketitle

\section{Introduction} \label{sec:intro}

A K\"ahler manifold $(M, J, \omega)$ is a smooth manifold $M$ of dimension $2n$ endowed with an integrable almost complex structure $J$ and a symplectic form
$\omega$ such that $g(X,Y)=\omega(X,JY)$ def{}ines a Riemannian metric, called \textit{K\"ahler metric}. In order to check that a compact manifold does not carry any K\"ahler metric,
one can use a collection of known topological obstructions to the existence of such a structure:
theory of K\"ahler groups, evenness of odd-degree Betti numbers, Lefschetz property or the formality of the rational homotopy type (see \cite{Amoros,DGMS,Wells}).

If $M$ is a compact K\"ahler manifold, then it has a complex and a 
symplectic structure. However, the converse is not true.
The f{}irst example of a compact manifold admitting complex and symplectic structures but no K\"ahler metric is the Kodaira-Thurston manifold \cite{Kod,Thurston}. 
This $4$-manifold is not simply connected (it is actually a nilmanifold)
hence the fundamental group plays a key role in this property. The classif{}ication of complex 
and symplectic nilmanifolds of dimension 6 was given by Salamon in \cite{Salamon}.
Generalizations to higher dimension $2n\geq 6$ of the Kodaira-Thurston manifold 
are the generalized Iwasawa manifolds considered in \cite{CoFG}.
Such manifolds have complex and symplectic structures but carry no K\"ahler metric.
Note that, in dimension $2$, every oriented surface admits a K\"ahler metric.

If one restricts attention to manifolds with trivial fundamental group, then every complex manifold of real dimension $4$ admits a K\"ahler structure. Indeed, by the Enriques-Kodaira classif{}ication 
\cite{Kod}, if $M$ is a complex surface whose f{}irst Betti number $b_{1}$ is even (this holds in particular when $b_1=0$), then $M$ is deformation equivalent to a K\"ahler surface (see also 
\cite[Theorem 3.1, page 144]{BHPV} for a direct proof of this fact). 
We point out that Gompf \cite{Gompf} has constructed the f{}irst examples of simply connected compact symplectic
but not complex $4$-manifolds. Also Fintushel and Stern \cite{Fintushel-Stern} have given a family of simply connected symplectic $4$-manifolds not admitting complex structures (the latter was 
proved by Park \cite{Park}).

In dimensions higher than 4, we have the following results. The f{}irst examples of simply connected compact symplectic 
non-K\"ahlerian manifolds were given in dimension 6 by Gompf in the 
aforementioned paper \cite{Gompf} and in dimension $\geq 10$ by McDuf{}f in \cite{McDuff} (these examples are not known 
to admit complex structures). Fine and Panov in \cite{FP} (see also \cite{FP_2}) have 
produced simply connected symplectic $6$-manifolds with $c_1=0$ which do not have a compatible complex structure 
(but it is not known if they admit K\"ahler structures). Furthermore, Guan in \cite{Gu} constructed the f{}irst family of simply connected, 
compact and holomorphic symplectic non-K\"ahlerian manifolds of (real) dimension $4n\geq 8$. On the other hand, the f{}irst and third authors have proved \cite{BaMu} 
that the $8$-dimensional manifold $X$ constructed in \cite{FM} is an
example of a simply connected, symplectic and complex manifold which does not admit a K\"ahler structure (since it is not formal). 
For higher dimensions $2n=8+2k$, $k\geq 1$, one can take $X\x \CP^k$. This is simply connected, complex and 
symplectic but not K\"ahler. Thus, a natural question arises:
\begin{quote}
\textit{Does there exist a 6-dimensional simply connected, compact, symplectic and complex manifold which does not admit K\"ahler metrics?}
\end{quote}
In this paper we answer this question in the af{}f{}irmative by proving the following result:

\begin{theorem}\label{main}
There exists a $6$-dimensional, simply connected, compact, symplectic and complex manifold 
which carries no K\"ahler metric.
\end{theorem}
In order to construct such an example, we start with a $6$-dimensional nilmanifold $M$ admitting both a complex structure $J$ and a symplectic structure $\omega$. Then we quotient it by a f{}inite 
group preserving $J$ and $\omega$ to obtain a simply connected, 
$6$-dimensional orbifold $\widehat M$ with an orbifold complex structure $\widehat J$ and
an orbifold symplectic form $\widehat\omega$. By Hironaka Theorem \cite{Hironaka},
there is a complex resolution $(\widetilde M_{c}, \widetilde J)$ of $(\widehat M, \widehat J)$. As in \cite{CFM}, we resolve symplectically the singularities of $(\widehat M, \widehat \omega)$ to 
obtain a smooth symplectic $6$-manifold $(\widetilde M_{s}, \widetilde \omega)$. 
However, in our situation, the singular locus of the orbifold $\widehat{M}$ does not consist
only of a discrete set of points, in contrast with \cite{CFM}. For a complex and symplectic orbifold, we provide conditions under which the 
complex and the symplectic resolution of singularities are dif{}feomorphic (Theorem \ref{thm:resolutions}). Using this we prove that the resolutions $\widetilde M_{c}$ and $\widetilde M_{s}$ are 
dif{}feomorphic. Thus, $\widetilde M=\widetilde M_{c}$ is not only a complex manifold but also a symplectic one.

Next, we show that $\widetilde M$ is simply connected (Proposition 6.1), this resulting from the careful choice of the action of the f{}inite group on $M$. Since $\widetilde M$ is a $6$-dimensional 
simply connected compact manifold, then $b_1(\widetilde M)=0$, and $b_3(\widetilde M)$ is even by Poincar\'e duality. Also $\widetilde M$ is automatically formal by \cite[Theorem 3.2]{FM2}. Therefore, 
to ensure that $\widetilde M$ does not carry any K\"ahler metric, we use the Lefschetz property; more precisely, we prove that the map $L_{[\Omega]}\colon H^2(\widetilde M)\to H^4(\widetilde 
M)$ given by the cup product with $[\Omega]$ is not an isomorphism for any possible symplectic form $\Omega$. Again the choice of nilmanifold $M$ and f{}inite group action makes possible to have a 
non-zero $[\beta]\in H^2(\widetilde M)$ such that $[\beta]\wedge [\alpha_1]\wedge [\alpha_2]=0$ for every $[\alpha_1], [\alpha_2]\in H^2(\widetilde M)$, which gives the result.

\noindent {\bf Acknowledgments.} We would like to thank Claude LeBrun and Dietmar Salamon for the conversations that they had with the second author during her stay at the
Simons Center for Geometry and Physics, Stony Brook University, in September 2014. We are grateful to Joel Fine, Robert Gompf, Daniel Guan, Emmy Murphy, Francisco Presas, S\"onke Rollenske and 
Richard Thomas for useful comments. Second author partially supported through Project MICINN (Spain) MTM2011-28326-C02-02, and Project of UPV/EHU ref.\ UFI11/52.

\section{Orbifolds} \label{sec:orbifolds}

\begin{definition}
A (smooth) $n$-dimensional orbifold is a Hausdorf{}f, paracompact topological 
 space $X$ endowed with an atlas $\cA=
\{(U_p, {\widetilde U}_p, \Gamma_p, \varphi_p)\}$ of orbifold charts,
that is $U_p\subset X$ is a neighbourhood of $p\in X$, $\widetilde U_p\subset \RR^n$ an
open set, $\Gamma_p\subset \GL(n, \RR)$ a f{}inite group acting on $\widetilde U_p$,
and $\varphi_p\colon{\widetilde U_p}\to U_p$ is a $\Gamma_p$-invariant map
with $\varphi_p(0)=p$, inducing a homeomorphism ${\widetilde U_p}/\Gamma_p \cong U_p$.

The charts are compatible in the following sense:
if $q \in U_q\cap U_p$, then there exist a connected 
neighbourhood $V\subset U_q \cap U_p$ and a 
dif{}feomorphism $f\colon\varphi_p^{-1}(V)_0 \to \varphi_q^{-1}(V)$,
where $\varphi_p^{-1}(V)_0$ is the connected component of $\varphi_p^{-1}(V)$ 
containing $q$, such that $f(\sigma(x))=\rho(\sigma)(f(x))$, for
any $x$, and $\sigma\in \Stab_{\Gamma_p}(q)$, where 
$\rho\colon \Stab_{\Gamma_p}(q) \to \Gamma_q$ is a group isomorphism.
\end{definition}

For each $p\in X$, let $n_p=\# \Gamma_p$ be the order of the orbifold point (if $n_p=1$ the point is smooth, also called non-orbifold point). The singular locus of the orbifold is the set $S=\{p\in X 
\ | \ n_p>1\}$. Therefore $M-S$ is a smooth $n$-dimensional manifold.
The singular locus $S$ is stratif{}ied: if we write $S_k=\{p \ | \ n_p=k\}$, and consider its closure $\overline{S_k}$, then $\overline{S_k}$ inherits the structure of an orbifold. In particular $S_k$ 
is a smooth manifold, and the closure consists of some points of $S_{kl}$, $l\geq 2$. 

We say that the orbifold is locally oriented if $\Gamma_p\subset \GL_+(n,\RR)$ for any $p\in X$. As $\Gamma_p$ is f{}inite, we can choose a metric on $\widetilde U_p$ such that $\Gamma_p\subset 
\SO(n)$. An element $\sigma\in \Gamma_p$ admits a basis in which it is written as
$\sigma=\diag \left(
\left( \begin{array}{cc} \cos \theta_1 & -\sin \theta_1\\ \sin \theta_1 & \cos \theta_1\end{array} \right),
\ldots, \left( \begin{array}{cc}  \cos \theta_r & -\sin \theta_r\\ \sin \theta_r & \cos \theta_r\end{array}\right), 1,\ldots, 1\right)$,
for $\theta_1,\ldots,\theta_r\in (0,2\pi)$. In particular, the set of points f{}ixed by $\sigma$ is of codimension $2r$. Therefore the set of singular points $S\cap U_p$ is of codimension $\geq 2$, 
and hence $X-S$ is connected (if $X$ is connected). Also we say that the orbifold $X$ is oriented if it is locally oriented and $X-S$ is oriented.

A natural example of orbifold appears when we take a smooth manifold $M$ and a f{}inite group $\Gamma$ acting on $M$ ef{}fectively. Then $\widehat{M}=M/\Gamma$ is an orbifold. If $M$ is oriented and 
the action of $\Gamma$ preserves the orientation, then $\widehat{M}$ is an oriented orbifold. Note that for every $\widehat p\in \widehat{M}$, the group $\Gamma_{\widehat p}$ is the stabilizer of 
$p\in M$, with $\widehat p=\widehat\pi(p)$ under the natural projection $\widehat \pi\colon M\to \widehat{M}$.

\begin{definition}\label{def:complex_orbifold}
A complex orbifold is a $2n$-dimensional orbifold $X$ whose orbifold charts have $\widetilde U_p\subset \CC^n$,
$\Gamma_p\subset \GL(n, \CC)$, and in the compatibility of charts the maps $f$ are biholomorphisms.
Note that $X$ is automatically oriented.
\end{definition}

If $M$ is a complex manifold and $\Gamma$ is a f{}inite group acting ef{}fectively
on $M$ by biholomorphisms, then $\widehat{M}=M/\Gamma$ is a complex orbifold. 

The complex structure of a complex orbifold $X$ can be given by the orbifold $(1,1)$-tensor $J$ 
with $J^2=-\id$. This is given by tensors $J_p$ on each $\widetilde U_p$ def{}ining the complex structure,
which are $\Gamma_p$-equivariant, for each $p\in X$, and which agree under the functions
$f$ def{}ining the compatibility of charts.

\begin{definition} \label{def:resolution_complex}
A {\it complex resolution} of a complex orbifold $(X,J)$ is a complex manifold
$\widetilde{X}$ together with a holomorphic map $\pi\colon\widetilde{X}\to X$ which is
a biholomorphism $\widetilde{X} - E \to X-S$,
where $S\subset X$ is the singular locus and $E=\pi^{-1}(S)$ is the {\it exceptional locus}.
\end{definition}

Let $X$ be an orbifold. An orbifold $k$-form $\alpha$ consists of a collection of $k$-forms $\alpha_p$ on 
each $\widetilde U_p$ which are $\Gamma_p$-equivariant and that match under
the compatibility maps between dif{}ferent charts. 

\begin{definition}\label{def:symplectic_orbifold}
A symplectic orbifold $(X, \omega)$ consists of a $2n$-dimensional oriented orbifold $X$ and an orbifold 
$2$-form $\omega$ such that $d\omega=0$ and $\omega^n>0$ everywhere.
\end{definition}

If $M$ is a symplectic manifold and $\Gamma$ is a f{}inite group acting ef{}fectively
on $M$ by symplectomorphisms, then $\widehat{M}=M/\Gamma$ is a symplectic orbifold. 

\begin{definition} \label{def:resolution_symplectic}
A {\it symplectic resolution} of a symplectic orbifold
$(X,\omega)$ consists of a smooth symplectic manifold
$(\widetilde{X},\widetilde{\omega})$ and a map $\pi\colon\widetilde{X}\to X$ such
that:
  \begin{itemize}
  \item $\pi$ is a dif{}feomorphism $\widetilde{X}- E \to X-S$,
  where $S\subset X$ is the singular locus and $E=\pi^{-1}(S)$ is
  the {\it exceptional locus}.
  \item $\widetilde\omega$ and $\pi^*\omega$ agree in the complement of a
  small neighbourhood of $E$.
  \end{itemize}
\end{definition}

\section{Desingularization of orbifold points}\label{sec:points}

In this section we suppose that $X$ is an oriented orbifold whose singular locus $S$ consists of a discrete set of points. Assume that $X$ admits a complex structure $J$ and a symplectic structure 
$\omega$. Therefore we have a complex orbifold $(X,J)$ and a symplectic orbifold $(X,\omega)$. 

It is well-known that $(X,J)$ admits a complex resolution $(\widetilde X_c,\widetilde J)$ by 
Hironaka's desingularization \cite{Hironaka}. Also, the symplectic orbifold
$(X,\omega)$ admits a symplectic resolution $(\widetilde X_s,\widetilde \omega)$ 
by Theorem 3.3 in \cite{CFM}. We want to compare the two resolutions.

First, let us look at the complex resolution of $(X,J)$.
Consider $p\in S$, and let $U_p=\widetilde U_p/\Gamma_p$ be an orbifold neighbourhood. 
Recall that we denote $\varphi_p\colon\widetilde U_p\to U_p$ the quotient map. 
By def{}inition of complex orbifold, $\widetilde U_p \subset \CC^n=\RR^{2n}$ and $\Gamma_p
\subset \GL(n,\CC)$. As $\Gamma_p$ is a f{}inite group, we can choose a K\"ahler metric 
invariant by $\Gamma_p$. With a linear change of variables, we can transform the K\"ahler metric into standard form. That is, 
we can suppose that 
there is an inclusion
\begin{equation}\label{inc:complex}
\imath\colon\Gamma_p\inc \U(n).
\end{equation}

Shrinking  
$\widetilde U_p$ if necessary, we can assume that $\widetilde U_p=B_\epsilon(0)$, for some $\epsilon>0$.

Consider now an algebraic resolution of the singularity of $Y=\CC^n/\Gamma_p$,
provided by \cite{Hironaka}. Denote it $\pi\colon\widetilde Y \to Y$, and let $E=\pi^{-1}(p)$
be the exceptional locus. Write $B=B_\epsilon(0)/\Gamma_p$ and 
$\widetilde B=\pi^{-1}(B)$. The complex resolution is def{}ined as the smooth manifold
\[
 \widetilde X_c = (X-\{p\} ) \cup_\pi \widetilde B,
\]
where we identify with
the map $\pi\colon\tilde B-E \to  B-\{p\}= U_p -\{p\}$.
This has a natural complex structure since $\pi$ is a biholomorphism.

Now we move to the construction of the symplectic resolution of $(X,\omega)$,
as done in \cite{CFM}. For $p \in S$, take an orbifold neighbourhood
$U_p'=\widetilde U_p'/\Gamma_p'$, with $\varphi_p'\colon\widetilde U_p'\to U_p'$. 
By the equivariant Darboux theorem, there
is a $\Gamma_p'$-equivariant symplectomorphism $(\widetilde U_p',\omega_p) \cong (V,\omega_0)$, where $V \subset \RR^{2n}$ is an open set, and $\omega_0$ is the standard symplectic form (shrinking $\widetilde U_p'$ if necessary). 
So without loss of generality, we can assume that 
$\widetilde U_p' \subset (\RR^{2n},\omega_0)$, where
$\omega_0$ is the standard symplectic form, and $\Gamma_p' \subset \Sp(2n,\RR)$. 
As $\Gamma_p'$ is a f{}inite group, and $\U(n)\subset \Sp(2n,\RR)$ is the maximal compact subgroup, we can choose a complex structure $J$ on $\RR^{2n}$ such
that the pair $(J, \omega_0)$ determines a K\"ahler metric, which is invariant by $\Gamma_p'$. 
We perform a linear change of variables, which transforms the complex structure 
into standard form (so $\widetilde U_p'$ has the standard K\"ahler structure). 
Equivalently, 
we can suppose that there is an inclusion
\begin{equation}\label{inc:symplectic}
\imath'\colon\Gamma_p'\inc \U(n).
\end{equation}

Shrinking $\widetilde U_p'$ if necessary, we can assume that $\widetilde U_p'=B_{\epsilon'}(0)$, for some $\epsilon'>0$.

Consider an algebraic resolution of singularities of $Y'=\CC^n/\Gamma_p'$,
say $\pi'\colon\widetilde Y' \to Y'$, and let $E'=(\pi')^{-1}(p)$
be the exceptional locus. Write $B'=B_{\epsilon'}(0)/\Gamma_p'$ and 
$\widetilde B'=(\pi')^{-1}(B')$. The symplectic resolution is def{}ined as the smooth manifold
\[
\widetilde X_s = (X-\{p\} ) \cup_{\pi'} \widetilde B',
\]
where $\widetilde B'-E'$ and $B'-\{p\}= U_p' -\{p\}$ are identif{}ied by $\pi'$.
This has a symplectic structure that is constructed by gluing the 
symplectic structure of $X-\{p\}$ and the K\"ahler form of $\widetilde B'$ by a 
cut-of{}f process, as done in Theorem 3.3 of \cite{CFM}.

Now we are going to compare $\widetilde X_c$ and $\widetilde X_s$. 
First note that for $p\in S$, we have $\Gamma_p\cong \Gamma_p'$. This follows
from $\Gamma_p\cong \pi_1(B-\{p\})$
and $\Gamma_p'\cong \pi_1(B'-\{p\})$, and the fact that $B, B'$ are homeomorphic. 
So we shall denote $\Gamma_p'=\Gamma_p$ henceforth.
We have the following result.

\begin{theorem} \label{thm:resolutions}
If one can arrange that the inclusions $\imath$ and $\imath'$, given by \eqref{inc:complex} and \eqref{inc:symplectic}, respectively, 
are such that
$\imath=\imath'$ for every singular point $p\in S$, then there is a dif{}feomorphism 
$\widetilde X_c  \cong \widetilde X_s$, which is the identity outside a small neighbourhood
of the exceptional loci. In particular, $\widetilde X_c$ admits both complex and
symplectic structures.
\end{theorem}

\begin{proof}
The key point is obviously that if $\imath=\imath'$, then $Y'=Y$, so
we can take $\widetilde Y'=\widetilde Y$ and $\pi'=\pi$ in the constructions above.

We f{}ix a point $p\in S$, and construct the required isomorphism in a neighbourhood
of the exceptional locus over that point. Consider the map (reducing $\epsilon>0$ if necessary)
\[
 f=(\varphi_p')^{-1}\circ \varphi_p \colon B_\epsilon(0)=\widetilde U_p \to B_{\epsilon'}(0)=\widetilde U_p';
\]
$f$ is $\Gamma_p$-equivariant and an open embedding (it might fail to be surjective) with $f(0)=0$.
We shall construct a map $F\colon B_\epsilon(0)\to B_{\epsilon'}(0)$ such that
 \begin{itemize}
 \item $F=\id$ in a small ball $B_{0.2\epsilon}(0)$, 
 \item $F=f$ outside a slightly bigger ball $B_{0.9\epsilon}(0)$, 
 \item $F$ is a $\Gamma_p$-equivariant dif{}feomorphism onto its image.
\end{itemize}
This gives a dif{}feomorphism $ F\colon \widetilde X_c \to \widetilde X_s$, def{}ined by $F$ on $B_\epsilon(0)/\Gamma_p-\{p\}$,
extended by the identity on $\pi^{-1}(B_{0.2\epsilon}(0)/\Gamma_p)$, and also by the identity 
on $X-\pi^{-1}(B_{0.9\epsilon}(0)/\Gamma_p)$.

Write $f(x)=L(x)+ R(x)$, where $L$ is the linear part and $|R(x)|\leq C|x|^2$, for some constant
$C>0$. Both these maps are $\Gamma_p$-equivariant.
Take a smooth, non-decreasing function $\rho_1\colon[0,\epsilon]\to [0,1]$ such that $\rho_1(t)=0$
for $t\in [0,0.8\epsilon]$ and $\rho_1(t)=1$ for $t\in [0.9\epsilon,1]$. Consider
$g(x)=L(x)+\rho_1(|x|) R(x)$. Then, $g(x)=L(x)$ for $|x|\leq 0.8\epsilon$,
$g(x)=f(x)$ for $|x|\geq 0.9\epsilon$, and $g(x)$ is $\Gamma_p$-equivariant because $\Gamma_p\subset \SO(2n)$.
Also
\[
 dg(x)-L =\rho_1'(|x|) R(x) d|x| + \rho_1(|x|) dR(x).
\]
Using that $|\rho_1'(t)|\leq C/\epsilon$ and $|dR(x)|\leq C|x|$ (we denote by $C>0$
uniform constants, that can vary from line to line)
we have that $|dg(x)-L|\leq C|x|$.
For $\epsilon>0$ small enough, we have that $g$ is a dif{}feomorphism onto its image.

For the next step, take the linear map $L\colon\RR^{2n}\to \RR^{2n}$. 
We can choose orthonormal (oriented) basis in both origin and target so that
$L=\diag(\lambda_1,\ldots, \lambda_{2n})$, where $\lambda_i>0$ are real numbers (the f{}irst vector of the basis
is a unitary vector $e_1$ such that $|L(e_1)|$ is maximized; then $L$ maps 
$\la e_1\ra^\perp$ to $\la L(e_1) \ra^\perp$, and we proceed inductively). Consider the map 
\[
 h(x)= \left\{ \begin{array}{ll} x, & |x| \leq 0.4 \epsilon, \\
 x + \rho_2\left(\left(\frac{|x|-0.4\epsilon}{0.3\epsilon}\right)^\alpha \right) 
  (L(x)-x), \qquad & 0.4 \epsilon \leq |x| \leq 0.7\epsilon, \\  
  g(x), & |x|\geq 0.7\epsilon, \end{array}\right.
\]
where $\rho_2\colon[0,1] \to [0,1]$ is smooth non-decreasing with $\rho_2(t)=0$ for $t\in [0,\frac13]$, and $\rho_2(t)=1$ for $t\in [\frac23,1]$. Here $\alpha>0$ is a constant to be f{}ixed soon. 

Clearly $h$ is $\Gamma_p$-equivariant, $h(x)=f(x)$ of{}f $B_{0.9\epsilon}(0)$,
and $h(x)=x$ in $B_{0.4\epsilon}(0)$ (but beware,
we have chosen dif{}ferent coordinates on the origin $\RR^{2n}$ and the target $\RR^{2n}$, so $h$ is not the identity in the ball).
The map $h$ is $\cC^\infty$
because for $0.4 \epsilon \leq |x| \leq 0.5\epsilon$ we have also $h(x)=x$. Let us see that $h$ is a dif{}feomorphism onto its image. 
It only remains to see this for $ 0.5 \epsilon \leq |x|\leq 0.7 \epsilon$.
Write $y=h(x)$, so in our coordinates 
$y_i=x_i+ \rho_2(u) (\lambda_i-1)x_i$, with $u=\left(\frac{|x|-0.4\epsilon}{0.3\epsilon}\right)^\alpha$. 
Then,
 \begin{align*}
dy_i =& \left(1+ (\lambda_i-1) \rho_2(u)\right) dx_i + 
(\lambda_i-1) \rho_2'(u) \frac{\alpha}{0.3\epsilon} \left(\frac{|x|-0.4\epsilon}{0.3\epsilon}\right)^{\alpha-1}  x_i \gamma
 \end{align*}
with 
$\gamma=d|x|=\frac{1}{|x|} \sum x_jdx_j$. Write $\delta_i =\left(1+ (\lambda_i-1)\rho_2(u)\right)$, so $\delta_i$ takes values between $1$ and $\lambda_i$. We compute
 \begin{align*}
 dy_1 \wedge \ldots \wedge  dy_n & = \delta_1\ldots \delta_n  \, dx_1 \wedge \ldots \wedge  dx_n + \\ &+
 \sum \delta_1\ldots \hat{\delta_i}\ldots \delta_n
 \frac{(\lambda_i-1) \rho_2'(u)\alpha x_i}{0.3\epsilon}\!\! \left(\frac{|x|-0.4\epsilon}{0.3\epsilon}\right)^{\alpha-1} 
  dx_1 \wedge \ldots \wedge \stackrel{(i)}{\gamma} \wedge \ldots \wedge dx_n \\
 &=  \delta_1\ldots \delta_n \left( 1+  \alpha \sum  \frac{ (\lambda_i-1) \rho_2'(u)
(|x|-0.4\epsilon)^{\alpha-1}  x_i^2 }{|x|\delta_i (0.3\epsilon)^\alpha} \right)dx_1 \wedge \ldots \wedge dx_n.
 \end{align*}
In the sum, the numerator is bounded above by $C(0.3\epsilon)^{\alpha+1}$ and
the denominator is bounded below by $C^{-1}(0.3\epsilon)^{\alpha+1}$,
for some uniform (independent of $\alpha$) constant $C>0$.
Hence choosing $\alpha>0$ small enough, we get that the above quantity does not vanish, 
hence $h$ is a dif{}feomorphism onto its image.

After this step is done, recall that we have taken coordinates given by an orthonormal basis
$\{e_i\}$ on the origin $\RR^{2n}$, and by the orthonormal basis
$\{L(e_i)/\lambda_i\}$ on the target $\RR^{2n}$. Written with respect to the same
coordinates, we have an orthogonal transformation $M\colon \RR^{2n}\to \RR^{2n}$
so that $h(x)=M$ on $B_{0.4\epsilon}(0)$. 
The f{}inal step is to change the isometry $M\in \SO(2n)$ by the identity. Take a smooth path $M_t$ of matrices joining $M_0=\id$ with $M_1=M$. Take a 
smooth non-decreasing $\rho_3\colon [0,\epsilon] \to [0,1]$ with $\rho_3(t)=0$ for $t\in [0,0.2\epsilon]$, and $\rho_3(t)=1$ for $t\in [0.3\epsilon,\epsilon]$. The map
$F(x)=M_{\rho_3(|x|)}(x)$, $|x|\leq 0.4\epsilon$, and $F(x)=h(x)$ for $|x|\geq 0.4\epsilon$,
is the required map.
\end{proof}

\begin{remark} \label{rem:kahler}
Let $F\colon (\widetilde{X}_c,\widetilde J) \to (\widetilde{X}_s,\widetilde \omega)$ be the 
dif{}feomorphism provided by Theorem \ref{thm:resolutions}. Then if we denote $\widetilde \omega'=F^*\widetilde \omega$,
we have that $\widetilde X_c$ admits a symplectic structure $\widetilde \omega'$ and 
a complex structure $\widetilde J$. These are not compatible in general, but
they are compatible on a neighbourhood of the exceptional locus, and give a K\"ahler structure there.
\end{remark}

\section{A complex and symplectic $6$-orbifold} \label{sec:complex-symplectic}

Consider the complex Heisenberg group $G$, that is, the
complex nilpotent Lie group of (complex) dimension 3 consisting of matrices of the form
\[
 \begin{pmatrix} 1&u_2&u_3\\ 0&1&u_1\\ 0&0&1\end{pmatrix}.
\]
In terms of the natural
(complex) coordinate functions $(u_1,u_2,u_3)$ on $G$, we have
that the complex $1$-forms $\mu=du_1$, $\nu=du_2$ and
$\theta=du_3-u_2 \, du_1$ are left invariant, and
\[
 d\mu=d\nu=0, \quad d\theta=\mu\wedge\nu.
\]
Let $\Lambda \subset \CC$ be the lattice generated by $1$
and $\zeta=e^{2\pi i/6}$,
and consider the discrete subgroup
$\Gamma\subset G$ formed by the matrices in which $u_1,u_2,u_3 \in \Lambda$. We def{}ine the compact (parallelizable)
nilmanifold
\[
 M=\Gamma \backslash G.
\]
We can describe $M$ as a principal torus bundle
\[
 T^2=\CC/\Lambda \inc M \to T^4=(\CC/\Lambda)^2
\]
by the projection $(u_1,u_2,u_3) \mapsto (u_1,u_2)$.

Consider the action of the f{}inite group $\ZZ_6$ on $G$ given by the generator
\begin{eqnarray*}
 \rho\colon G & \to & G\\
 (u_1, u_2,u_3) &\mapsto & (\zeta^4\, u_1, \zeta\, u_2,\zeta^5\, u_3) .
\end{eqnarray*}
This action satisf{}ies that $\rho(p\cdot q)=\rho(p)\cdot \rho(q)$,
for $p, q \in G$, where $\cdot$ denotes the natural group
structure of $G$. Moreover, $\rho(\Gamma)=\Gamma$. Thus, $\rho$
induces an action on the quotient $M=\Gamma\backslash G$. Denote by $\rho\colon M \to M$ the $\ZZ_6$-action. The action on 1-forms is given by
\[
 \rho^*\mu= \zeta^4\, \mu, \quad
 \rho^*\nu= \zeta\, \nu,\quad
 \rho^*\theta= \zeta^5\,\theta.
\]

\begin{proposition}
$\widehat M=M/\ZZ_{6}$ is a $6$-orbifold admitting complex and symplectic structures.
\end{proposition}

\begin{proof}
The nilmanifold $M$ is a complex manifold whose complex structure $J$ is
the multiplication by $i$ at each tangent space $T_{p}M$, $p\in M$. 
Then one can check that $J$ commutes with the
$\ZZ_6$-action $\rho$ on $M$, that is, $(\rho_{*})_p\circ J_{p} = J_{\rho(p)}\circ (\rho_{*})_p$,
for any point $p\in M$. Hence, $J$ induces a complex structure on the quotient $\widehat M=M/\ZZ_{6}$.

Now we def{}ine the complex $2$-form $\omega$ on $M$ given by
 \begin{equation}\label{eqn:omega}
 \omega=- i \,\mu \wedge \bmu + \nu\wedge \theta + \bnu \wedge \btheta.
 \end{equation}
Clearly, $\omega$ is a real closed $2$-form on $M$ which satisf{}ies $\omega^3 > 0$,
that is, $\omega$ is a symplectic form on $M$. Moreover, $\omega$ is 
$\ZZ_{6}$-invariant. Indeed, 
$ \rho^*\omega=- i \,\mu \wedge \bmu + \zeta^6 \nu\wedge \theta +\zeta^{-6}\bnu \wedge \btheta=\omega$.
Therefore $\widehat{M}$ is a {symplectic $6$-orbifold}, with the symplectic form
$\widehat\omega$ induced by $\omega$. 
\end{proof}

We denote by
\[
 \widehat\pi\colon M\to \widehat M
\]
the natural projection. The orbifold points of $\widehat{M}$ are the following:
\begin{enumerate}
\item The points $(\frac13 a(1+\zeta), \frac13  b(1+\zeta),\frac13 c(1+\zeta)+\frac29 ab(1+\zeta)^2)\in M$, $a,b,c\in \{0,1,2\}$ and $(b,c)\neq (0,0)$, are points of order $3$, with 
isotropy group $K=\{\id,\rho^2,\rho^4\}$. These points are mapped in pairs by
$\ZZ_6$, so they def{}ine 12 orbifold points in $\widehat M=M/\ZZ_{6}$, with models $\CC^3/K$.
\item The surfaces $S_{(p,q)}=\{(u_1, p, p \, u_1 + q) \ | \ u_1\in \CC/\Lambda\}\subset M$,
where $p,q\in \{0, \frac12,\frac\zeta2,\frac{1+\zeta}2\}$, $(p,q)\neq (0,0)$. These are
$15$ tori, which consist of points of order $2$, with
isotropy $H=\{\id,\rho^3\}$. These surfaces are permuted by the group $\ZZ_6$, so 
they come in $5$ groups of three tori each. Thus they def{}ine $5$ tori in the 
orbifold $\widehat{M}$, formed by orbifold points of order $2$.

\item The surface $S_0=\{(u_1,0,0) \ | \ u_1\in \CC/\Lambda\}\subset M$ is a torus consisting generically of points of order $2$, with isotropy $H$. Here $\rho\colon S_0\to S_0$ and
it is a map of order $3$, with three f{}ixed points $(\frac13 a(1+\zeta), 0,0)$, $a=0,1,2$.
These points have isotropy $\ZZ_6$. The quotient $S_0/\la\rho\ra \subset \widehat{M}$ is homeomorphic to a sphere (with three orbifold points of order $3$).
\end{enumerate}

\section{Resolution of the $6$-orbifold} \label{sect:resolution}

Now we want to desingularize the orbifold $\widehat{M}$. We shall treat each
of the connected components of the singular locus determined before independently. Recall that $K=\{\id,\rho^2,\rho^4\}\cong \ZZ_3$ and $H=\{\id,\rho^3\}\cong \ZZ_2$.
There is a natural isomorphism $\la \rho \ra =\ZZ_6 \cong K\x H$.

\subsection{Resolution of the isolated orbifold points}

We know that there are 12 isolated orbifold points in $\widehat{M}$. Let $\widehat p\in \widehat M$ be one of them. 
The preimage of $\widehat p$ under $\widehat\pi$  consists of two points, $\widehat{\pi}^{-1}(\widehat p)=\{p_1, p_2\}$. The isotropy group of $p_1$ is $K$. Consider a $K$-invariant
neighbourhood $U$ of $p_1$ in $M$. Then, 
\[
 \widehat U=\widehat\pi(U) \cong U/K
\]
is an orbifold neighbourhood of $\widehat p$ in $\widehat M$.
This has complex and symplectic resolutions as in
Section \ref{sec:points}. In order to apply Theorem 
\ref{thm:resolutions} we check that $\imath=\imath'\colon K\to \U(3)$.
For the complex resolution, we have $\imath(\zeta^2)=
\diag (\zeta^2,\zeta^2, \zeta^4)$. For the symplectic
resolution, the symplectic form (\ref{eqn:omega}) is, in our coordinates $(u_1,u_2,u_3)$,
 \begin{equation}\label{eqn:omega2}
 \omega=- i \,du_1 \wedge d\bar u_1+ du_2\wedge du_3+ d\bar u_2\wedge d \bar u_3\, .
 \end{equation}
We have to do a change of variables to transform $K \subset \Sp(6,\RR)$
into a subgroup of $\U(3)$. This is obtained with 
 \begin{align*}
v_1 &= u_1 \\ 
v_2 &= \frac1{\sqrt2} (u_2 -i \bar u_3) \\
v_3 &= \frac1{\sqrt2} (\bar u_2 -i u_3 ).
 \end{align*}
This transforms (\ref{eqn:omega2}) into
\[
 \omega=- i \,dv_1 \wedge d\bar v_1 - i \,dv_2 \wedge d\bar v_2 - i \,dv_3 \wedge d\bar v_3,
\]
the standard K\"ahler form. 
The $K$-action is given by $(v_1,v_2,v_3)\mapsto (\zeta^2 v_1,\zeta^2 v_2, \zeta^4 v_3)$, so 
$\imath' (\zeta^2)=\diag (\zeta^2,\zeta^2, \zeta^4)$, and $\imath=\imath'$.

\subsection{Resolution of the singular sets $\widehat\pi(S_{(p,q)})$}

Now we consider a connected component of the singular set which is homeomorphic to a $2$-torus. There are $5$ such components in $\widehat{M}$, all of them are images by $\widehat\pi$ of the sets $S_{(p,q)}=\{(u_1, p, p \, u_1 + q) \ | \ u_1\in \CC/\Lambda\}$,
where $(p,q)\in I=\left(\{0, \frac12,\frac\zeta2,\frac{1+\zeta}s\}\right)^2-\{ (0,0)\}$.

Let us focus on one such component $\widehat T=\widehat\pi(T)$, 
$T\cong \CC/\Lambda$.
Then $H$ f{}ixes $S_{(p,q)}$,
and its orbit under $K$ is given by $S_{(p_i,q_i)}$, for three elements $(p_1,q_1)=(p,q), (p_2,q_2),(p_3,q_3) \in I$.
Consider a neighbourhood $U$ of $T\subset M$ via
\begin{eqnarray*}
 T \x B_\epsilon(0) & \to & U \\
 (u_1,u_2,u_3) &\mapsto&  (u_1,u_2+ p, u_3+ p \, u_1 + q) ,
\end{eqnarray*}
where $B_\epsilon(0)\subset \CC^2$. The image is
\begin{equation}\label{eqn:U}
\widehat U=\widehat\pi(U) \cong U/H \cong T \x (B_\epsilon(0)/H),
\end{equation}
where $H\cong \ZZ_2$ acts as $(u_2,u_3)\mapsto (-u_2,-u_3)$.

We see that the complex structure on (\ref{eqn:U}) is the product complex
structure. Also, the symplectic structure 
$\omega= i \,du_1\wedge d\bar u_1 + du_2\wedge du_3+ d\bar u_2 \wedge d\bar u_3$
is the product of the natural symplectic structure of $\CC/\Lambda$ with an orbifold
symplectic structure on $B_\epsilon(0)/H$.
Using the construction of Section \ref{sec:points}, we have a desingularization 
\[
\widetilde Y \to B_\epsilon(0)/H
\]
which is a smooth manifold endowed with both a complex structure 
and a symplectic structure coinciding with the given ones outside a 
small neighbourhood of the exceptional locus $E$. The condition $\imath=\imath'$ of
Theorem \ref{thm:resolutions} is trivially satisf{}ied, since $\imath(\rho^3)=\imath'(\rho^3)=-\id$.
Multiplying by $T=\CC/\Lambda$, we have that
\[
\widetilde U=T \x \widetilde Y
\] 
is a smooth manifold endowed with a complex structure $\widetilde J$,
and a symplectic structure $\widetilde \omega$, which 
coincide with those of $\widehat U$ outside a small neighbourhood of 
the exceptional locus $T\x E\subset \widetilde U$.

The complex and the symplectic resolutions of $\widehat M$ in a neighbourhood of $\widehat T$ are obtained by replacing
$\widehat U\subset \widehat M$ with $\widetilde U$.
The two resolutions are dif{}feomorphic by the considerations above.

\subsection{Resolution of the singular set $\widehat\pi(S_0)$}

Finally we consider the connected component of the singular set which is homeomorphic to a $2$-sphere.
This is $\widehat S_0=\widehat\pi(S_0)$, where $S_{0}=\{(u_1, 0,0) \ | \ u_1\in \CC/\Lambda\}$. As before, a neighbourhood of $S_0$ in $M$ is of the form 
\[
 U_0=(\CC/\Lambda) \x B_\epsilon(0),
\]
where $B_\epsilon(0)\subset \CC^2$. The action of $H=\ZZ_2$ is 
trivial on $\CC/\Lambda$ and as $\pm 1$ on $\CC^2$. 
The action of $K=\ZZ_3$ is of the form $\rho^2 (u_1,u_2,u_3)=(\zeta^2u_1,\zeta^2 u_2, \zeta^4 u_3)$. 

Let us focus on $B_\epsilon(0)/H$. By the construction of Section \ref{sec:points}, we have a 
complex desingularization $(\widetilde Y_c,\widetilde J) \to B_\epsilon(0)/H$. The
holomorphic action of $K$ on $B_\epsilon(0)$ induces an action on 
$(\widetilde Y_c,\widetilde J)$. Also, there is a symplectic desingularization 
$(\widetilde Y_s,\widetilde \omega) \to B_\epsilon(0)/H$. The action of $K$ on 
$B_\epsilon(0)$ induces an action on $(\widetilde Y_s,\widetilde \omega)$.
This follows by taking an orbifold chart of the singular point
that is $(H\x K)$-equivariant, using the equivariant Darboux theorem.

By Theorem \ref{thm:resolutions}, there is a dif{}feomorphism $F\colon(\widetilde Y_c,\widetilde J)\to (\widetilde Y_s,\widetilde \omega)$. Let us see that $F$ can be taken to 
be $K$-equivariant. This follows by the arguments in the proof of Theorem \ref{thm:resolutions}
by using that $\imath\colon H\x K \to \U(2)$ and $\imath'\colon H \x K \to \U(2)$ are equal. 
For the complex case, $\imath$ is given by the representation 
$(u_2,u_3)\mapsto (\zeta u_2, \zeta^5 u_3)$, so 
$\imath (\zeta)=\diag (\zeta, \zeta^5)$. For the symplectic case,
we have to do a change of variables to transform $H\x K \subset \Sp(4,\RR)$
into a subgroup of $\U(2)$. This is given by 
\[
v_2 = \frac1{\sqrt2} (u_2 -i \bar u_3),\quad v_3 = \frac1{\sqrt2} (\bar u_2 -i u_3 ),
\]
which transforms $\omega= du_2\wedge du_3+ d\bar u_2\wedge d \bar u_3$ into the
standard K\"ahler form $- i \,dv_2 \wedge d\bar v_2 - i \,dv_3 \wedge d\bar v_3$.
As $(v_2,v_3)\mapsto (\zeta v_2, \zeta^5 v_3)$, we have that 
$\imath' (\zeta)=\diag (\zeta, \zeta^5)$. Hence $\imath=\imath'$.

This produces a desingularization $\widetilde Y \to B_\epsilon(0)/H$ with a symplectic and a complex structure, which match the given ones outside a small neighbourhood of
the exceptional set $E\subset \widetilde Y$, which are compatible (they give a K\"ahler
structure) in a smaller neighbourhood of $E$, by Remark \ref{rem:kahler}, and which
have an action of $K$ preserving both the complex and symplectic structures.
A desingularization of 
\[
 U_0/H = (\CC/\Lambda) \x (B_\epsilon(0)/H)
\]
is given by substituting a neighbourhood of $\widehat S_0=(\CC /\Lambda) \x \{0\}$ by 
$(\CC /\Lambda) \x \widetilde Y$. The f{}ixed points of action of $K$ in $U_0/H$ lie 
on $\widehat S_0$, hence the f{}ixed points of the action of $K$ on the desingularization
of $U_0/H$ lie in the exceptional divisor. In this part of the manifold, we have
a K\"ahler structure, so the symplectic and complex desingularization are the same.

This means that $(U_0/H)/K \cong U_0/(H\x K)$ admits a desingularization $\widetilde V$ with
a complex and a symplectic structure. The resolution of $\widehat M$ in a neighbourhood of $\widehat S_0$
is obtained by substituting 
$\widehat\pi(U_0)=U_0/(H\x K) \subset \widehat M$ with $\widetilde V$.

All together, we get a smooth $6$-manifold $\widetilde M$
with a complex structure and a symplectic structure, and with a map
\[
 \pi\colon\widetilde M \too \widehat M,
\]
which is simultaneously a complex and a symplectic resolution.

\section{Topological properties of $\widetilde M$}

In this section, we are going to complete the proof of Theorem \ref{main} by proving
that $\widetilde M$ is simply-connected and that it does not admit a K\"ahler structure.

\begin{proposition}\label{prop:1-connected}
$\widetilde{M}$ is simply connected.
\end{proposition}

\begin{proof} 
We f{}ix base points $p_0 =(0,0,0)\in M$ 
and $\widehat p_0=\widehat\pi(p_0) \in \widehat M$. There is an
epimorphism of fundamental groups
\[
 \Gamma=\pi_1(M,p_0) \surj \pi_1(\widehat M,\widehat p_0),
\]
since the $\ZZ_6$-action has a f{}ixed point \cite[Chapter II, Corollary 6.3]{Bredon}. 
Now the nilmanifold $M$ is a principal $2$-torus
bundle over the $4$-torus $T^4$, so we have an exact sequence
\[
 \ZZ^2 \inc \Gamma \to \ZZ^4.
\]
The group $\Gamma=\pi_1(M,p_0)$ is thus generated by the images of the fundamental groups of 
the surfaces $\Sigma_1=\{(u_1,0,0)\}$, $\Sigma_2=\{(0,u_2,0)\}$ and $\Sigma_3=\{(0,0,u_3)\}$ in $M$. The
image 
$\widehat\pi(\Sigma_1)$ is a $2$-sphere, since 
$\widehat\pi\colon \Sigma_1\to \widehat\pi(\Sigma_1)$ is a degree $3$ map with three 
ramif{}ication points of order $3$ (namely $(\frac12 a(1+\zeta),0,0)$, with $a=0,1,2$). The image of $\Sigma_2$
is also a $2$-sphere, since $\widehat\pi\colon \Sigma_2\to\widehat\pi(\Sigma_2)$ is a degree $6$ map with one point of order $6$, $(0,0,0)$, two of order $3$, $(0,\frac12 b(1+\zeta),0)$, $b=1,2$, and three of order $2$ (namely $(0,p,0)$, $p=\frac12,\frac\zeta2,\frac{1+\zeta}2$).
Analogously, $\widehat\pi(\Sigma_3)$ is a $2$-sphere. This proves that $\pi_1(\widehat M, \widehat p_0)=\{1\}$.

Now we look at the resolution process. Let $S\subset\widehat M$ be the singular locus and suppose $p\in S$ is an isolated orbifold point. The
resolution replaces a neighbourhood $B=B_\epsilon(0)/\Gamma_p$ of $p$ with a smooth manifold $\widetilde B$, such that $\pi\colon\widetilde B\to B$ is a complex resolution of singularities. The manifold $\widetilde B$ is simply connected by \cite[Theorem 4.1]{Verbitsky}.
A Seifert-Van Kampen argument gives that $\pi_1(\widehat M)$ is the amalgamated sum of
$\pi_1(\widehat M-\{p\})$ and $\pi_1(B)$ along $\pi_1(\bd B)$. Also 
 $\pi_1(\widetilde M)$ is the amalgamated sum of
$\pi_1(\widetilde M-E)$ and $\pi_1(\widetilde B)$ along $\pi_1(\bd B)$. As $\pi_1(B)=\pi_1(\widetilde B)=\{1\}$,
we have that $\pi_1(\widehat M)=\pi_1(\widetilde M)$.

Suppose now that we have a connected component $S'$ of
the singular locus $S$ of positive dimension. Let $E'=\pi^{-1}(S')$ be the 
corresponding exceptional locus. The invariance of the fundamental group under resolution
is proved along the same lines as before if
we know that the map $\pi\colon E' \to S'$ induces an isomorphism $\pi_1(E')\to \pi_1(S')$.
In our case, we have two possibilities: if 
$S'=\widehat\pi(S_{(p,q)}) \cong T^2$, then $E'= T^2 \x E$,
where $E$ is the exceptional divisor of the resolution $\widetilde Y \to B_\epsilon(0)/H$, which is clearly simply connected, and the result follows.

The second possibility is $S'=\widehat\pi(S_0)$. In this case, the exceptional divisor over $S'$ is the exceptional divisor of the resolution of 
\[((\CC/\Lambda) \x (\CC^2/H))/K.
\]
The resolution of $\CC^2/H$ is done by blowing-up $\CC^2$ at the origin,
\[
 \widetilde\CC^2=\{(a,b, [u:v])\in \CC^2 \x \CP^1 \ | \ a v= b u \},
\]
and then quotienting by $H=\{\pm \id\}$. Clearly, the fundamental groups of
$(\CC/\Lambda) \x (\CC^2/H)$ and $(\CC/\Lambda) \x (\widetilde\CC^2/H)$ coincide.
The action of $K$ is given by $(a,b,[u:v])\mapsto ((\zeta^2 a, \zeta^4 b),[u:\zeta^2 v])$,
with f{}ixed points $(0,0,[1:0])$ and $(0,0,[0:1])$ 
The f{}ixed points of $K$ on $((\CC/\Lambda) \x ( \widetilde\CC^2/H)$ occur 
when $K$ f{}ixes both factors. Therefore, all f{}ixed points are isolated, and the
second resolution does not alter the fundamental group.
\end{proof}

In order to prove that $\widetilde{M}$ does not admit a K\"ahler structure, we are going to check that it does not satisfy the Lefschetz condition \emph{for any symplectic form}. For this, it is necessary to understand the cohomology $H^*(\widetilde{M})$. 

We start by computing the cohomology of $\widehat{M}$. 
By Nomizu theorem \cite{No}, the cohomology of the nilmanifold $M$ is:
\begin{align*}
 H^0(M,\CC) = \ & \la 1\ra,\\
 H^1(M,\CC) = \ & \la [\mu], [\bar\mu],[\nu],[\bar\nu]\ra,\\
 H^2(M,\CC) = \ & \la [\mu \wedge \bar\mu], [\mu \wedge\bar \nu],  [\bar\mu \wedge \nu], [\nu\wedge \bar\nu],
[\mu\wedge \theta], [ \bar\mu\wedge \bar\theta],  [\nu\wedge\theta],[ \bar\nu\wedge \bar\theta]\ra,\\
  H^3(M,\CC) = \ & \la [\mu\wedge\bar\mu\wedge \theta], [\mu\wedge\bar\mu\wedge \bar\theta],
 [\nu\wedge\bar\nu\wedge \theta],  [\nu\wedge\bar\nu\wedge \bar\theta],
 [\mu\wedge\nu\wedge \theta],  [\bar\mu \wedge \bar\nu \wedge \bar\theta]\\
 &  [\mu \wedge \bar\nu \wedge \theta],  [\mu \wedge \bar\nu \wedge \bar\theta],
 [\bar\mu \wedge \nu \wedge   \theta],[\bar\mu \wedge \nu \wedge \bar\theta]\rangle,\\
H^4(M,\CC) = \ & \la [\mu \wedge \bar\mu \wedge \nu \wedge \theta],
   [\mu \wedge \bar\mu \wedge \bar\nu \wedge \bar\theta],
  [\bar\mu \wedge \nu \wedge \bar\nu \wedge \bar\theta],
 [\mu \wedge \nu \wedge \bar\nu \wedge \theta],\\
  & [\mu \wedge \bar\mu \wedge \theta \wedge \bar\theta],
   [\nu \wedge \bar\nu \wedge \theta \wedge \bar\theta],
   [\mu \wedge \bar\nu \wedge \theta \wedge \bar\theta],
   [\bar\mu \wedge \nu \wedge \theta \wedge \bar\theta]\ra ,
 \end{align*}
\begin{align*}
 H^5(M,\CC) = \ & \la [\mu \wedge \bar\mu \wedge \nu \wedge
  \theta \wedge \bar\theta],
 [\mu \wedge \bar\mu \wedge \bar\nu \wedge
  \theta \wedge \bar\theta],
  [\mu \wedge \nu \wedge \bar\nu \wedge
  \theta \wedge \bar\theta], [\bar\mu \wedge \nu \wedge \bar\nu \wedge
  \theta \wedge \bar\theta] \ra, \\
  H^6(M,\CC) =\ & \la [\mu \wedge \bar\mu \wedge \nu \wedge \bar\nu
  \wedge \theta \wedge \bar\theta]\ra.
\end{align*}
So the cohomology $H^*(\widehat{M})=H^*(M)^{\ZZ_6}$ of $\widehat{M}$ is:
\begin{align*}
H^0(\widehat{M},\CC) = \ & \la 1\ra,\\
H^1(\widehat{M},\CC) = \ &0,\\
H^2(\widehat{M},\CC) = \ & \la [\mu \wedge \bar\mu],  [\nu\wedge \bar\nu],
 [\nu\wedge\theta],[\bar\nu\wedge \bar\theta]\ra,\\
H^3(\widehat{M},\CC) = \ & 0,\\
H^4(\widehat M,\CC) = \ & \la [\mu \wedge \bar\mu \wedge \nu \wedge \theta],
   [\mu \wedge \bar\mu \wedge \bar\nu \wedge \bar\theta],
 [\mu \wedge \bar\mu \wedge \theta \wedge \bar\theta],
   [\nu \wedge \bar\nu \wedge \theta \wedge \bar\theta]\ra ,\\
H^5(\widehat{M},\CC) = \ & 0 , \\
H^6(\widehat{M},\CC) = \ & \la [\mu \wedge \bar\mu \wedge \nu \wedge \bar\nu
  \wedge \theta \wedge \bar\theta]\ra.
\end{align*}
 
\begin{proposition}\label{non-Lefschetz}
$\widetilde{M}$ does not admit a K\"ahler structure since it does not satisfy the Lefschetz property for any symplectic form on $\widetilde{M}$.
\end{proposition}

\begin{proof}
Let $\Omega$ be a symplectic form on $\widetilde{M}$. The Lefschetz map $L_{[\Omega]} \colon H^2(\widetilde M) \to H^4(\widetilde M)$ is given by the cup product with $[\Omega]$. We show that there 
is a class $[\beta]\in H^2(\widetilde M)$ which is in the kernel of $L_{[\Omega]}$. We prove this by checking that $[\Omega]\wedge [\beta]\wedge [\alpha]=0$, for any $2$-form $[\alpha]\in 
H^2(\widetilde M)$.

We need to determine the cohomology $H^2(\widetilde M)$. For this, the f{}irst step is to construct a map $H^2(\widehat M) \to H^2(\widetilde M)$. Let $h\colon M\to M$ be a map which:
 \begin{itemize}
  \item is the identity outside small neighbourhoods of each point with non-trivial isotropy,
  \item contracts a neighbourhood of each of the isolated 24 points with isotropy $K$ onto the corresponding point,
 \item contracts a neighbourhood of each $S_{(p,q)}$ onto $S_{(p,q)}$ (f{}ixing $S_{(p,q)}$ pointwise),
 \item in a neighbourhood of $S_0$, is the composition of a contraction onto $S_0$ with a map 
   that contracts neighbourhoods (in $S_0$) of the $3$ f{}ixed points to the points, and
\item is $\ZZ_6$-equivariant.
\end{itemize}

$h$ induces a map $\widehat h\colon \widehat M\to\widehat M$. 
Note that for any closed form $\alpha \in \Omega^*(\widehat M)$, $\widehat h^*(\alpha) \in \Omega^*(\widehat M)$ is cohomologous to $\alpha$ and can be lifted to a form $\pi^*\widehat 
h^*(\alpha)\in\Omega^*(\widetilde M)$, where $\pi\colon\widetilde M \to \widehat M$ is the resolution map. This induces a well-def{}ined map
\[
\Psi= \pi^*\circ \widehat h^*\colon H^*(\widehat M) \to H^*(\widetilde M).
\]
Now consider $U= \widehat M-S$, where $S\subset\widehat M$ is the singular locus
and $V\subset \widehat M$ is a small neighbourhood of $S$. Let also $\widetilde U=\pi^{-1}(U)$ and $\widetilde V=\pi^{-1}(V) \subset \widetilde M$. Using compactly supported de Rham cohomology, we have a diagram
\[
  \begin{array}{cccccccccccccccc}
  H_c^2(U)\oplus H_c^2(V)  &\to & H_c^2(\widehat M) &\to & H_c^3(U\cap V) & \to & H_c^3(U)\oplus H_c^3(V) \\
   \downarrow=\quad \Psi\downarrow && \,\,\downarrow\Psi && \downarrow\cong && \downarrow=\quad \Psi\downarrow \\
  H_c^2(\widetilde U)\oplus H_c^2(\widetilde V) &\to & 
   H_c^2(\widetilde M) &\to & H_c^3(\widetilde U\cap \widetilde V) & \to & H_c^3(\widetilde U)\oplus H_c^3(\widetilde V) \\
  \end{array}
\]
Since $V$ retracts onto a set of dimension $2$, $H^3(V)=0$. By Poincar\'e duality, $H^3_c(V)=0$ as well.
Now a simple diagram chasing proves that $H^2(\widetilde M)=H^2_c(\widetilde M)$ is generated by $H^2(\widehat M)=
H^2_c(\widehat M)$ and 
$H^2_c(\widetilde V)$.

Consider the closed form $\nu\wedge \bnu\in \Omega^2(\widehat M)$.
Since $\nu\wedge\bnu|_{S_{(p,q)}}=0$ for any surface $S_{(p,q)}$ and 
$\nu\wedge\bnu|_{S_0}=0$ as well, the $2$-cohomology class 
\[
[\beta]=\Psi([\nu\wedge\bnu])
\] 
vanishes on $\widetilde V$. Clearly $[\beta]\wedge [\alpha_1]\wedge [\alpha_2]=0$ if either
$[\alpha_1],[\alpha_2]\in H_c^2(\widetilde V)$.
Moreover, one can check that $[\beta]\wedge [\alpha_1]\wedge [\alpha_2]=0$, for $[\alpha_1],[\alpha_2]\in H^2(\widehat{M})$, which completes the proof.
\end{proof}


\end{document}